\documentclass[11pt]{article}
\usepackage[T1]{fontenc}
\usepackage{lmodern}
\usepackage{amsmath,amssymb,amsthm,mathtools}
\usepackage[margin=1in]{geometry}
\usepackage[hidelinks]{hyperref}
\usepackage{tikz}
\usetikzlibrary{arrows.meta}

\newtheorem{theorem}{Theorem}[section]
\newtheorem{question}[theorem]{Question}
\newtheorem{problem}[theorem]{Problem}
\newtheorem{lemma}[theorem]{Lemma}
\newtheorem{proposition}[theorem]{Proposition}
\newtheorem{corollary}[theorem]{Corollary}

\theoremstyle{definition}
\newtheorem{definition}[theorem]{Definition}

\numberwithin{equation}{section}
\usepackage{bm}
\newcommand{\one}{\mathbf 1}
\newcommand{\cF}{\mathcal F}

\title{Oriented Paths with Few Direction Flips Are Tournament Anti-Sidorenko}
\author{Hao Chen\thanks{School of Mathematical Sciences, Soochow University,
Suzhou 215006, China. E-mail:
\href{mailto:chenhao@suda.edu.cn}{\texttt{chenhao@suda.edu.cn}}.}
\and
Yupeng Lin\thanks{School of Mathematical Sciences,
University of Science and Technology of China, Hefei, Anhui 230026, China. Research supported by National Key Research and Development Program of China 2023YFA1010201 and National Natural Science Foundation of China (Grant No. 125B2019 and 12125106). E-mail:
\href{mailto:lyp_inustc@mail.ustc.edu.cn}{\texttt{lyp\_inustc@mail.ustc.edu.cn}}.}}
\date{}

\begin{document}

\maketitle

\begin{abstract}
An oriented graph \(H\) is said to be \emph{tournament anti-Sidorenko (TAS)} if a uniformly random tournament asymptotically maximizes the homomorphism density of \(H\) among all tournaments. For an oriented path \(P\), a \emph{direction flip} is a non-leaf source or sink. Sah, Sawhney and Zhao proved that consistently directed paths (paths with no direction flips) are TAS. He, Mani, Nie, Tung and Wei proved that for \(3\le k\le 7\), every oriented path of length \(k\) with exactly one direction flip is TAS, and Chen, Clemen and Noel recently extended this to every \(k\ge 3\). In this paper, we extend these results by proving that for every integer \(r\ge 0\), every oriented path of length at least \(1665r+1454\) with \(r\) direction flips is tournament anti-Sidorenko. 
\end{abstract}

\section{Introduction}

Let \(H\) and \(G\) be graphs. A \emph{graph homomorphism} from \(H\) to \(G\) is a map \(f:V(H)\to V(G)\) such that for every edge \(uv\in E(H)\), the pair \(f(u)f(v)\) is an edge of \(G\). The \emph{homomorphism density} of \(H\) in \(G\) is
\[
t(H,G):=\frac{|\operatorname{Hom}(H,G)|}{|V(G)|^{|V(H)|}},
\]
where \(\operatorname{Hom}(H,G)\) denotes the set of all homomorphisms. Sidorenko's conjecture~\cite{Sidorenko93} asserts that for every bipartite graph \(H\) and every graph \(G\),
\[
t(H,G)\ge t(K_2,G)^{|E(H)|}.
\]
Equivalently, among graphs \(G\) of fixed edge density \(p=\frac{2|E(G)|}{|V(G)|^2}\), the density of \(H\) is asymptotically minimized by the Erd{\"o}s-R{\'e}nyi random graph \(G(n,p)\). This conjecture has been intensively studied and verified for many families of graphs. We refer the reader to  \cite{ConlonKimLeeLee18,ConlonLee17,ConlonLee21,Coregliano24,ILL25} for some recent progress in Sidorenko's conjecture.

In this paper we study an analogous question in the setting of tournaments. An \emph{oriented graph} \(H=(V(H),E(H))\) consists of a vertex set $V(H)$ and \(E(H)\subseteq V(H)\times V(H)\) of directed edges with no loops and at most one directed edge between each pair of distinct vertices. A \emph{tournament} is an oriented complete graph. For an oriented graph \(H\) and a tournament \(T\), a homomorphism from $H$ to $T$ is a map \(f:V(H)\to V(T)\) such that every edge \((u,v)\in E(H)\) maps to an edge \((f(u),f(v))\in E(T)\). The homomorphism density is then defined naturally as
\[
t(H,T):=\frac{|\operatorname{Hom}(H,T)|}{|V(T)|^{|V(H)|}}.
\]
A natural question is whether certain oriented graphs satisfy a Sidorenko‐type inequality in tournaments. An interesting result proved by Sah, Sawhney and Zhao in~\cite{SahSawhneyZhao23} implies that for every tournament \(T\), $t(P_k,T)\le (1/2)^{k},$
where $P_k$ is the consistently directed $k$-edge path. This motivates the following two definitions (see~\cite{FoxHimwichManiZhou24+,FoxHimwichManiZhou25}).

\begin{definition}
An oriented graph \(H\) is \emph{tournament Sidorenko (TS)} if for every tournament \(T\),
\[
t(H,T)\ge (1-o(1))(1/2)^{|E(H)|},
\]
where the \(o(1)\) term tends to zero as \(|V(T)|\to\infty\).
\end{definition}
 
\begin{definition}
An oriented graph \(H\) is \emph{tournament anti-Sidorenko (TAS)} if for every tournament \(T\),
\[
t(H,T)\le (1/2)^{|E(H)|}.\]
\end{definition}

Thus the directed path \(P_k\) is TAS. However, for general oriented paths the situation is more subtle. Zhao and Zhou~\cite{ZhaoZhou20} characterized all \emph{impartial} digraphs and hence, in particular, all \emph{impartial} paths, which are both TS and TAS. He, Mani, Nie, Tung and Wei~\cite{HeManiNieTungWei25+} classified the TS/TAS property for some short oriented paths. A natural question is whether some structural parameter determines the TS/TAS property of a given oriented path.

We encode an oriented path with vertex set $\{v_0,v_1,\ldots,v_k\}$ and \(k\) edges by a vector \(\vec{x}=(x_1,\dots,x_k)\in\{-1,1\}^k\), where \(x_i=1\) means the edge is oriented from \(v_{i-1}\) to \(v_i\), and \(x_i=-1\) means from \(v_i\) to \(v_{i-1}\). We denote this path by \(P_k(\vec{x})\). A \emph{direction flip} is an internal vertex \(v_i\) with \(x_i x_{i+1}=-1\). Equivalently, it is an internal source or sink. The flips and the two
endpoints partition the edge positions into maximal constant-sign blocks. For example, \(P_3(1,1,-1)\) has one flip at $v_2$, while \(P_6(1,1,-1,-1,1,1)\) has two flips at \(v_2\) and \(v_4\), splitting the path into three blocks of lengths \(2,2,2\).
\begin{figure}[htbp]
\centering
\hspace*{-0.065\textwidth}% 整体左移量，可根据需要调整
\begin{minipage}[c]{0.30\textwidth}
\centering
\scalebox{0.72}{%
\begin{tikzpicture}[
  vertex/.style={circle, draw=black, fill=white, inner sep=2pt,
                 minimum size=5mm, font=\small},
  flip/.style={circle, draw=red, thick, fill=red!10, inner sep=2pt,
               minimum size=6mm, font=\small},
  edge/.style={-{Stealth[length=2.5mm]}, thick}
]
% Left: P_3(1,1,-1)
\node[vertex] (v0) at (0,0) {$v_0$};
\node[vertex] (v1) at (1.8,0) {$v_1$};
\node[flip]  (v2) at (3.6,0) {$v_2$};
\node[vertex] (v3) at (5.4,0) {$v_3$};
\draw[edge] (v0) -- (v1);
\draw[edge] (v1) -- (v2);
\draw[edge] (v3) -- (v2);
\node[below=2mm] at (v2.south) {sink};
\end{tikzpicture}}%
\end{minipage}%
\quad
\begin{minipage}[c]{0.30\textwidth}
\centering
\scalebox{0.72}{%
\begin{tikzpicture}[
  vertex/.style={circle, draw=black, fill=white, inner sep=2pt,
                 minimum size=5mm, font=\small},
  flip/.style={circle, draw=red, thick, fill=red!10, inner sep=2pt,
               minimum size=6mm, font=\small},
  edge/.style={-{Stealth[length=2.5mm]}, thick}
]
% Right: P_6(1,1,-1,-1,1,1)
\node[vertex] (v0) at (0,0) {$v_0$};
\node[vertex] (v1) at (1.5,0) {$v_1$};
\node[flip]   (v2) at (3.0,0) {$v_2$};
\node[vertex] (v3) at (4.5,0) {$v_3$};
\node[flip]   (v4) at (6.0,0) {$v_4$};
\node[vertex] (v5) at (7.5,0) {$v_5$};
\node[vertex] (v6) at (9.0,0) {$v_6$};
\draw[edge] (v0) -- (v1);
\draw[edge] (v1) -- (v2);
\draw[edge] (v3) -- (v2);
\draw[edge] (v4) -- (v3);
\draw[edge] (v4) -- (v5);
\draw[edge] (v5) -- (v6);
\node[below=2mm] at (v2.south) {sink};
\node[below=2mm] at (v4.south) {source};
\end{tikzpicture}}%
\end{minipage}

\caption{Direction flips in $P_3(1,1,-1)$ and $P_6(1,1,-1,-1,1,1)$.
The red vertices are direction flips: $v_2$ is a sink in the left path,
and $v_2,v_4$ are a sink and a source respectively in the right path.}
\label{fig:direction-flip-example}
\end{figure}
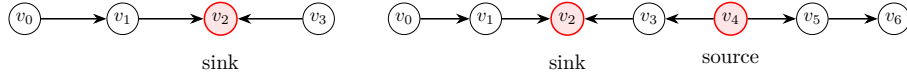

He, Mani, Nie, Tung and Wei~\cite{HeManiNieTungWei25+} proved that for \(3\le k\le 7\), every oriented path of length $k$ and exactly one direction flip is TAS, and they conjectured that for \(k\ge 3\), every oriented path of length $k$ and exactly one direction flip is TAS. Recently, Chen, Clemen and Noel~\cite{ChenFelixNoel26+} confirmed this conjecture and further showed that if every block of an oriented path has length at least two and every internal block has length divisible by four, then the oriented path is TAS.

It is natural to ask whether, when the number of direction flips is ``small'' compared with the length of the oriented path, the corresponding orientation must be anti-Sidorenko. 
In this paper, we confirm this phenomenon by proving the following theorem.

\begin{theorem}\label{thm:main}
There exist constants \(\alpha,\beta>0\) such that for all integers $k>r\geq0$ with \(k\ge \alpha r+\beta\), every oriented path of length $k$ containing $r$ direction flips is tournament anti-Sidorenko. 
In particular, one can take $\alpha=1665$ and $\beta=1454$.
\end{theorem}

For an integer \(k\ge1\), let \(f(k)\) denote the maximal \(m \in \{0,1,\ldots,k\}\)
such that every orientation of the \(k\)-edge path with fewer than \(m\)
direction flips is tournament anti-Sidorenko. The following corollary is immediate from Theorem~\ref{thm:main}.
\begin{corollary}\label{coro:main}
For every integer $k\ge 1454$,
\[
\displaystyle
 f(k)\ge
 1+\left\lfloor\frac{k-1454}{1665}\right\rfloor .
 \label{eq:f-bound}
\]
Moreover, \[\liminf_{k \to +\infty}\frac{f(k)}{k}\ge \frac{1}{1665}.\]
\end{corollary}

One might ask whether, dually, there exists some absolute constant $c$ such that for every sufficiently large $k$ and vector $\vec{x}\in \{-1,1\}^k$, oriented paths $P_k(\vec{x})$ with at least $ck$ direction flips are TS. In fact, no such constant $c$ exists. We will give a brief discussion of this in Concluding Remarks.

\section{Preliminaries}
\label{sec:preliminaries}

We recall some basic notations. 
Let $[k]:=\{1,\ldots,k\}$.
If $H$ is an
oriented graph, then $V(H)$ and $E(H)$ denote its vertex and edge
sets. When we say edges in $H$, we always mean directed edges. Throughout the paper we identify an oriented path of length \(k\) with its sign vector \(\vec{x}=(x_1,\dots,x_k)\in\{-1,1\}^k\); the \(i\)-th entry records the direction of the \(i\)-th edge relative to the natural order of the vertices $\{v_0,v_1,...,v_k\}$. We denote the corresponding oriented path by \(P_k(\vec{x})\). For integer $j>0$, the notation $\vec{\bm1_j}$ denotes the length-$j$ vector $(1,1,...,1)$. 

A direction flip in $P_k(\vec{x})$ is an internal vertex $v_i$ with $x_{i}x_{i+1}=-1.$ Note that the number of direction flips is determined by the vector $\vec{x}$. Therefore if there are $r$ direction flips in $P_k(\vec{x})$, we also say $\vec{x}$ has $r$ direction flips.

For a subset $F\subseteq E(H)$, let $H\langle F\rangle$ be the
spanning subgraph of $H$ with edge set $F$, that is, the vertex set
is $V(H)$ and the edge set is $F$. 

\subsection{Graph limits and tournamentons}\label{subsec:graphlimits}

In this paper, we work within the framework of graph limits. 
For a comprehensive introduction to graph limits, we refer
readers to~\cite{Lovasz12}. 
A \emph{kernel} is a bounded measurable function
$U\colon[0,1]^2\to\mathbb R$. A \emph{tournamenton} is a kernel
$W\colon[0,1]^2\to[0,1]$ such that $W(x,y)+W(y,x)=1$
for every $(x,y)\in[0,1]^2$. A kernel $U$ is
\emph{antisymmetric} if $U(x,y)=-U(y,x)$ for every $(x,y)\in[0,1]^2$. In particular, if $W$ is a
tournamenton, then $U=W-1/2$ is an antisymmetric kernel. We refer readers to~\cite{Grzesik+23,GrzesikKralLovaszVolec23,Kral+26+,NoelRanganathanSimbaqueba26} for a comprehensive treatment of the theory of tournamentons.

For an oriented graph $H$ and a kernel $U$, the \emph{homomorphism
density} of $H$ in $U$ is defined by
\begin{equation}
 t(H,U)=
 \int_{[0,1]^{|V(H)|}}
 \prod_{(v,w)\in E(H)} U(x_v,x_w)\,\prod_{u \in V(H)}dx_u.
 \label{eq:tournamenton-density}
\end{equation}

A sequence $(G_n)_{n\in\mathbb N}$ of tournaments is
\emph{convergent} if the sequence $t(H,G_n)$ converges for every
oriented graph $H$. By the standard theory of tournament limits, for
every convergent sequence of tournaments there exists a tournamenton
$W$ such that $t(H,G_n)\to t(H,W)$ for every oriented graph $H$; conversely, every tournamenton $W$
arises as the limit of some convergent sequence of tournaments. 

We now state and prove the equivalence between universal inequalities
over finite tournaments and over tournamentons. This allows us to
work exclusively with tournamentons in the analytic argument. 

\begin{proposition}
\label{prop:equivalence-prelim}
Let $H$ be an oriented graph and let $\xi\ge0$. The following two
statements are equivalent:
\begin{enumerate}
\item[(i)] $t(H,T)\le\xi$ for every finite tournament $T$.
\item[(ii)] $t(H,W)\le\xi$ for every tournamenton $W$.
\end{enumerate}
\end{proposition}

\begin{proof}
First assume that \emph{(ii)} holds. Let $T$ be a finite tournament
of order $n$. Partition $[0,1]$ into $n$ measurable sets
$A_1,\dots,A_n$ of equal measure. Define a tournamenton $W_T$ by
setting $W_T(x,y)=1$ and $W_T(y,x)=0$ whenever $x\in A_i$,
$y\in A_j$, and $i\to j$ is an edge of $T$; inside each diagonal
part $A_i\times A_i$, set $W_T=1/2$. For any map
$\phi\colon V(H)\to[n]$, the contribution to $t(H,W_T)$ coming from
the product region $\prod_{v\in V(H)}A_{\phi(v)}$ is
$n^{-|V(H)|}$ if $\phi$ is a homomorphism of $H$ into $T$, and is $0$
if $\phi$ is injective but not a homomorphism. Non-injective maps
contribute a nonnegative error bounded by
$1-\frac{(n)_{|V(H)|}}{n^{|V(H)|}}$, where $(n)_i:=\{n(n-1)\cdots(n-i+1)\}$. 
Hence
\begin{equation}
 t(H,T)\le t(H,W_T)
 \le t(H,T)+1-\frac{(n)_{|V(H)|}}{n^{|V(H)|}}.
 \label{eq:step-ineq}
\end{equation}
Since $t(H,W_T)\le\xi$ by assumption, we obtain $t(H,T)\le\xi$.
Thus \emph{(ii)} implies \emph{(i)}.

Conversely, assume that \emph{(i)} holds. Let $W$ be a
tournamenton. Sample independent uniform points
$X_1,\dots,X_n\in[0,1]$. Conditional on these points, orient each
unordered pair $\{i,j\}$ independently according to
$\mathbb P(i\to j\mid X_1,\dots,X_n)=W(X_i,X_j)$. This yields a
random finite tournament $T_n$ on vertex set $[n]$. For an injective $\phi\colon V(H)\to[n]$, the variables
$X_{\phi(v)}$ ($v\in V(H)$) are independent uniform on $[0,1]$ and the
edges are oriented independently, so
\[
  \mathbb P\bigl(\phi\text{ is a homomorphism }H\to T_n\bigr)
  =\mathbb E\Bigl[\prod_{(v,w)\in E(H)}
    W(X_{\phi(v)},X_{\phi(w)})\Bigr]
  =t(H,W).
\]
Hence the expected contribution of $\phi$ to
$t(H,T_n)=|\mathrm{Hom}(H,T_n)|/n^{|V(H)|}$ is $t(H,W)/n^{|V(H)|}$.  Summing over the
$(n)_{|V(H)|}$ injective maps and writing $C_n$ for the (nonnegative)
contribution of the non-injective maps, which is at most
$(n^{|V(H)|}-(n)_{|V(H)|})/n^{|V(H)|}$, gives
\[
  \mathbb E\,t(H,T_n)=\frac{(n)_{|V(H)|}}{n^{|V(H)|}}t(H,W)+C_n,
\]
where $0\le C_n\le 1-\frac{(n)_{|V(H)|}}{n^{|V(H)|}}.$
Since every
realization satisfies $t(H,T_n)\le\xi$, taking expectations gives
\[
 \frac{(n)_{|V(H)|}}{n^{|V(H)|}}t(H,W)\le\xi.
\]
Letting $n\to\infty$ proves $t(H,W)\le\xi$. Thus \emph{(i)} implies
\emph{(ii)}.
\end{proof}

\subsection{Spectral and Functional Analytic Tools}\label{subsec:spectral}

We collect here the standard operator-theoretic facts used in the
analytic part of the proof. We refer readers to~\cite{Conway} for more details. Throughout, \(L^2=L^2([0,1];\mathbb R)\)
with inner product
\[
 \langle f,g\rangle=\int_0^1 f(x)g(x)\,dx.
\]

A square-integrable kernel \(L\in L^2([0,1]^2)\) defines a bounded
linear operator on \(L^2\) by
\[
 (Lf)(x)=\int_0^1 L(x,y)f(y)\,dy.
\]
Such operators are Hilbert--Schmidt and hence compact; their adjoint
has kernel \(L^*(x,y)=L(y,x)\). An operator is \emph{positive semidefinite} if it is self-adjoint (\(T^*=T\)) and satisfies \(\langle Tg,g\rangle\ge0\) for every \(g\in L^2\).

We next recall some spectral properties for compact self-adjoint positive semidefinite operator. Let $A$ be a compact self-adjoint positive semidefinite operator. By the spectral theorem for compact self-adjoint positive semidefinite operators
\cite[Chapter II, \S5, Theorem 5.1 and Corollary 5.3]{Conway},
the nonzero spectrum of \(A\) consists of positive eigenvalues of
finite multiplicity, with zero as the only possible accumulation
point. Moreover, the eigenspaces corresponding to distinct nonzero
eigenvalues are orthogonal.

If $P_\lambda$ is the orthogonal
projection onto $\ker(A-\lambda I)$, then $
A=\sum_{\lambda>0}\lambda P_\lambda
$
with convergence in operator norm, whereas
$
I=P_0+\sum_{\lambda>0}P_\lambda
$
holds in the strong operator topology. Hence $\mathbf1=P_0\mathbf1+\sum_{\lambda>0}P_\lambda\mathbf1
$ with convergence in $L^2$. Define
$$
\mu_A=
\|P_0\mathbf1\|_2^2\delta_0+
\sum_{\lambda>0}\|P_\lambda\mathbf1\|_2^2\delta_\lambda,
$$
where $\delta_\lambda$ denote the Dirac
measure at $\lambda$, i.e. the probability measure supported on the single point $\lambda$.
The orthogonality of $P_\lambda$ and $\|\mathbf1\|_2=1$ show that $\mu_A$ is a
probability measure.

The following lemma is a specialized version of the spectral theorem. We refer readers to~\cite[Chapter IX]{Conway} for a complete treatment of this theorem.
\begin{lemma}
\label{lem:spectral-integration}
Let \(A\) be a compact self-adjoint positive semidefinite operator on
\(L^2[0,1]\), and let \(\mu_A\) be the probability measure on
\(\sigma(A)\) defined by $
\mu_A
=
\|P_0\one\|_2^2\,\delta_0
+
\sum_{\lambda>0}\|P_\lambda\one\|_2^2\,\delta_\lambda,
$
where \(P_0\) is the projection onto \(\ker A\), and \(P_\lambda\) is
the projection onto \(\ker(A-\lambda I)\). Then for every continuous
function \(f:\sigma(A)\to\mathbb R\),
\[
\langle \one,f(A)\one\rangle
=
\int_{\sigma(A)} f(\lambda)\,d\mu_A(\lambda).
\]
\end{lemma}

\begin{proof}
Since \(A\) is compact and positive semidefinite, its nonzero spectrum consists
of positive eigenvalues \(\lambda>0\) accumulating only at \(0\).
By the spectral theorem for compact self-adjoint operators, for
every continuous function \(f:\sigma(A)\to\mathbb R\),
\[
f(A)
=
f(0)I
+
\sum_{\lambda>0}
\bigl(f(\lambda)-f(0)\bigr)P_\lambda.
\]
The continuity of \(f\) at \(0\) gives
\(f(\lambda)-f(0)\to0\), so the series converges in operator norm.
Thus \(f(A)\) is bounded. Applying \(f(A)\) to the \(L^2\)-convergent
expansion
$
\one
=
P_0\one
+
\sum_{\lambda>0}P_\lambda\one$
and using the orthogonality of the spectral projections, we get $
f(A)\one
=
f(0)P_0\one
+
\sum_{\lambda>0}f(\lambda)P_\lambda\one.$
Taking the inner product with \(\one\) gives
\[
\langle \one,f(A)\one\rangle
=
f(0)\|P_0\one\|_2^2
+
\sum_{\lambda>0}f(\lambda)\|P_\lambda\one\|_2^2
=
\int_{\sigma(A)} f(\lambda)\,d\mu_A(\lambda).
\]
\end{proof}

We then give several basic lemmas from functional analysis that will be applied in Section~\ref{sec:large}. Readers may skip these lemmas for now and return to them when they are needed in Section~\ref{sec:large}. The proofs of these lemmas are standard in functional analysis, so we place the proofs in Appendix~\ref{app:A} so as not to interrupt the
main argument. 

A bounded linear operator \(T\) on \(L^2\) is
\emph{positive} if \(Tf\ge0\) almost everywhere whenever \(f\ge0\)
almost everywhere. It is important to highlight that here the notion ``positive” should not be confused with ``positive semidefinite”. 

For positive compact operators, we use the total-cone version of the classical Krein--Rutman theorem~\cite{KreinRutman1948} (see~\cite[Theorem 5.2]{LuLu2022}) to obtain the following lemma. A brief introduction of the Krein--Rutman theorem and why it implies Lemma~\ref{thm:positive-compact} is given in Appendix~\ref{app:krein-rutman}. 

\begin{lemma}[Krein--Rutman Theorem for $L^2$] 
\label{thm:positive-compact}
Let \(T\) be a positive compact operator on 
\(L^2[0,1]\). If its spectral radius \(\rho(T)\) is positive, then
there is a nonzero function \(f\ge0\) such that $Tf=\rho(T)f.$
\end{lemma}

The next lemma is the invertibility for a skew-adjoint perturbation. The proof of this lemma is in Appendix~\ref{app:skew}.
\begin{lemma}
\label{lem:skew-invertibility}
Let \(L^2=L^2([0,1];\mathbb R)\). If $C$ is a bounded operator
satisfying \(C^*=-C\) and \(\eta>0\), then \(\eta I-C\) is invertible, and $\|(\eta I-C)g\|_2^2
=\eta^2\|g\|_2^2+\|Cg\|_2^2$
for $g\in L^2.$
Furthermore,
\[
(\eta I-C)^{-1}
=(\eta I+C)(\eta^2I-C^2)^{-1}.
\]
\end{lemma}

We then introduce a lemma on the convergence of the Neumann series, the proof can be found in~\cite[\S5: 5.6.9. Neumann Series Expansion Theorem] {Kutateladze96}. For completeness, we give the proof of it in Appendix~\ref{app:Neumann}.
\begin{lemma}
\label{lem:Neumann}
Let \(T\) be a bounded real operator and let \(\omega>0\). If \(\rho(\omega T)<1\), then $(I-\omega T)^{-1}=\sum_{n\ge0}(\omega T)^n$
with convergence in operator norm.
\end{lemma}

Finally, we record the inverse formula for a rank-one perturbation.
Since its proof is short, we include it here.  
\begin{lemma}
\label{lem:rank-one}
Let \(h,g\in L^2\), \(\theta\in\mathbb R\), and define
$$
(h\otimes g)v=\langle v,g\rangle h.
$$
Then \(I-\theta\,h\otimes g\) is invertible if and only if $
1-\theta\langle h,g\rangle\neq0.$
In this case,
\[
(I-\theta\,h\otimes g)^{-1}
=
I+\frac{\theta}{1-\theta\langle h,g\rangle}\,h\otimes g.
\]
\end{lemma}

\begin{proof}
Let \(T=h\otimes g\) and \(c=\langle h,g\rangle\). Then
$T^2=\langle h,g\rangle T=cT.$

If \(1-\theta c\neq0\), set \(\beta=\theta/(1-\theta c)\). Since
\(T^2=cT\),
$
(I-\theta T)(I+\beta T)
=
I+(\beta-\theta)T-\theta\beta cT
=
I.$
The two factors commute, so \(I-\theta T\) is invertible with
inverse \(I+\beta T\).

Conversely, if \(1-\theta c=0\), then \(h\neq0\) and
$
(I-\theta T)h
=
h-\theta\langle h,g\rangle h
=
(1-\theta c)h
=
0.
$
Thus \(I-\theta T\) is not injective and hence not invertible.
\end{proof}

\subsection{Outward-oriented even paths and spectral moments}\label{subsec:moments}
For a tournamenton \(W\), write \(W=\frac12+U\) as kernels. We identify a square-integrable kernel with
the Hilbert--Schmidt operator it defines on \(L^2[0,1]\). Under this
identification, set
$C=2U$ and $K=2W,$
so that $K=J+C,$
where $(Jg)(x)=\langle g,\one\rangle\one(x)$ and $\one(x)=1.$
The kernel \(U\) is antisymmetric and takes values in
\([-1/2,1/2]\). Thus \(C\) is a compact skew-adjoint operator: $C^*=-C.$
The operator $A=-C^2=C^*C$ is compact, self-adjoint, and positive semidefinite. The operator $K$ is compact and positive. The following sharp spectral-radius bound from
\cite[Proposition~2.3]{Grzesik+23} plays a crucial role in the proof of this paper. 

\begin{theorem}[Proposition~2.3 in~\cite{Grzesik+23}]
\label{thm:published-radius}
If \(U:[0,1]^2\to[-1/2,1/2]\) is an antisymmetric kernel, then the
spectral radius of the operator \(U^2\) is at most \(1/\pi^2\).
\end{theorem}
Since \(C=2U\), Theorem~\ref{thm:published-radius} implies
\(\sigma(A)\subseteq[0,4/\pi^2]=[0,a]\), where \(a=4/\pi^2\).

For an integer \(s\ge 1\), let \(Q_{2s}\) denote the orientation of
the \(2s\)-edge path in which every edge is directed from the
central vertex toward one of the two endpoints. Thus \(Q_{2s}\)
consists of two directed paths of length \(s\) sharing the same
initial vertex. In this subsection, we will introduce some key properties of the homomorphism density of $Q_{2s}$, which can also be found in~\cite[\S2]{Grzesik+23}.

\begin{figure}[htbp]
\centering

% Q_2
\begin{minipage}[c]{0.18\textwidth}
\centering
\begin{tikzpicture}[
  vertex/.style={circle, fill=black, inner sep=0pt, minimum size=3pt,
                 outer sep=1pt},
  edge/.style={-{Stealth[length=2mm]}, thick},
  lab/.style={font=\small}
]
\node[vertex, label={[lab]below:$c$}]  (c)  at (0,0) {};
\node[vertex, label={[lab]above:$l_1$}] (l1) at (-0.7,0.7) {};
\node[vertex, label={[lab]above:$r_1$}] (r1) at (0.7,0.7) {};

\draw[edge] (c) -- (l1);
\draw[edge] (c) -- (r1);
\end{tikzpicture}\\[1.5mm]
$Q_2$
\end{minipage}%
\hspace{1em}
% Q_4
\begin{minipage}[c]{0.26\textwidth}
\centering
\begin{tikzpicture}[
  vertex/.style={circle, fill=black, inner sep=0pt, minimum size=3pt,
                 outer sep=1pt},
  edge/.style={-{Stealth[length=2mm]}, thick},
  lab/.style={font=\small}
]
\node[vertex, label={[lab]below:$c$}]    (c)  at (0,0) {};
\node[vertex, label={[lab]above:$l_1$}]  (l1) at (-0.65,0.6) {};
\node[vertex, label={[lab]above:$l_2$}]  (l2) at (-1.3,1.2) {};
\node[vertex, label={[lab]above:$r_1$}]  (r1) at (0.65,0.6) {};
\node[vertex, label={[lab]above:$r_2$}]  (r2) at (1.3,1.2) {};

\draw[edge] (c) -- (l1);
\draw[edge] (l1) -- (l2);
\draw[edge] (c) -- (r1);
\draw[edge] (r1) -- (r2);
\end{tikzpicture}\\[1.5mm]
$Q_4$
\end{minipage}%
\hspace{1em}
% Q_{2s}
\begin{minipage}[c]{0.42\textwidth}
\centering
\begin{tikzpicture}[
  vertex/.style={circle, fill=black, inner sep=0pt, minimum size=3pt,
                 outer sep=1pt},
  edge/.style={-{Stealth[length=2mm]}, thick},
  lab/.style={font=\small}
]
\node[vertex, label={[lab]below:$c$}]    (c)  at (0,0) {};

% left arm: first two vertices
\node[vertex, label={[lab]above:$l_1$}]  (l1) at (-0.55,0.5) {};
\node[vertex, label={[lab]above:$l_2$}]  (l2) at (-1.1,1.0) {};
% left ellipsis
\node[lab] (ldots) at (-1.85,1.7) {$\cdots$};
% left last vertex
\node[vertex, label={[lab]above:$l_s$}]  (ls) at (-2.6,2.4) {};

% right arm: first two vertices
\node[vertex, label={[lab]above:$r_1$}]  (r1) at (0.55,0.5) {};
\node[vertex, label={[lab]above:$r_2$}]  (r2) at (1.1,1.0) {};
% right ellipsis
\node[lab] (rdots) at (1.85,1.7) {$\cdots$};
% right last vertex
\node[vertex, label={[lab]above:$r_s$}]  (rs) at (2.6,2.4) {};

% left edges
\draw[edge] (c) -- (l1);
\draw[edge] (l1) -- (l2);
\draw[edge, dashed] (l2) -- (ldots);
\draw[edge, dashed] (ldots) -- (ls);
% right edges
\draw[edge] (c) -- (r1);
\draw[edge] (r1) -- (r2);
\draw[edge, dashed] (r2) -- (rdots);
\draw[edge, dashed] (rdots) -- (rs);
\end{tikzpicture}\\[1.5mm]
$Q_{2s}$
\end{minipage}

\caption{The paths \(Q_2\), \(Q_4\), and \(Q_{2s}\). All edges are directed
from the central vertex \(c\) toward the two endpoints.}
\label{fig:Q-paths-labeled}
\end{figure}
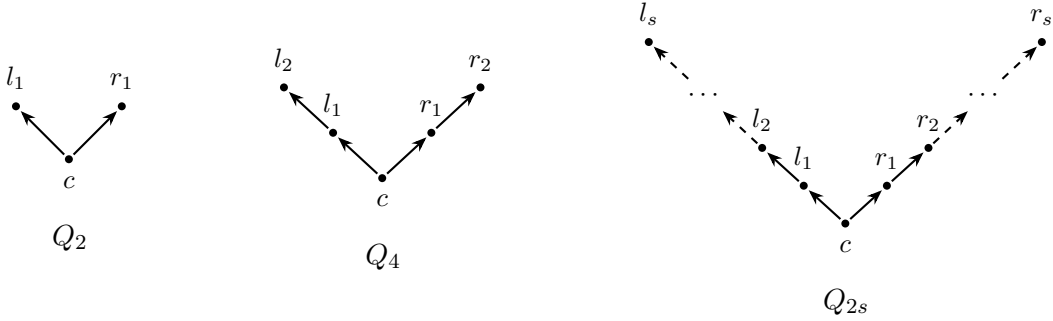

We now relate the homomorphism density of \(Q_{2s}\) in an
antisymmetric kernel \(U\) to the \(L^2\)-norm of the iterated
operator \(U^s\) applied to the constant function \(\one\). Let the
central vertex of \(Q_{2s}\) be labelled \(0\), and let the two arms
be labelled \(1,2,\dots,s\) and \(1',2',\dots,s'\), respectively,
with edges directed away from \(0\). Then, by the definition of the
homomorphism density,
\[
 t(Q_{2s},U)
 =
 \int_{[0,1]^{2s+1}}
 \prod_{i=0}^{s-1} U(x_i,x_{i+1})
 \prod_{j=0}^{s-1} U(y_j,y_{j+1})
 \,dx_0\,dx_1\cdots dx_s\,dy_1\cdots dy_s,
\]
where we set \(x_0=y_0\) as the common central variable, \(x_1,\dots,x_s\)
along the first arm, and \(y_1,\dots,y_s\) along the second arm.

For fixed \(x_0\), the integral over the first arm is exactly the function \((U^s\mathbf1)(x_0)\) applied to the constant
function:
\[
 \int_{[0,1]^s}
 \prod_{i=0}^{s-1} U(x_i,x_{i+1})
 \,dx_1\cdots dx_s
 =
 (U^s\one)(x_0),
\]
where \(\one\) denotes the function identically equal to \(1\).
Similarly, the integral over the second arm gives another factor
\((U^s\one)(x_0)\). Hence $t(Q_{2s},U)
 =
 \int_0^1 (U^s\one)(x_0)^2\,dx_0
 =
 \|U^s\one\|_2^2.$

For the normalized kernel \(C=2U\), we therefore define
\[
d_s(W)
:=\|C^s\one\|_2^2
=4^s\|U^s\one\|_2^2
=4^s t(Q_{2s},U)
=4^s t(Q_{2s},W-\frac{1}{2}).
\]
In particular, \(d_s(W)\ge0\).

We now express \(d_s(W)\) as a spectral moment. Since \(C^*=-C\), we have $d_s(W)=\langle C^s\one,C^s\one\rangle
=\langle \one,(C^*)^sC^s\one\rangle
=\langle \one,A^s\one\rangle .$
By Theorem~\ref{thm:published-radius}, $\sigma(A) \subseteq [0,a]$. By taking $f=\lambda^s$ in Lemma~\ref{lem:spectral-integration}, there is a probability measure \(\mu_A\) on \([0,a]\),
induced by the constant function \(\one\), such that for $s\ge 1$,
\[
d_s(W)=\int_0^a \lambda^s\,d\mu_A(\lambda).
\]
The following lemma from Grzesik, Il'kovi\v{c}, Kielak and Kr\'a\v{l}~\cite[Lemma 2.4 and Lemma 3.2]{Grzesik+23} is useful when bounding $d_s(W)$.

\begin{lemma}[Lemma 2.4 and Lemma 3.2 in \cite{Grzesik+23}]
\label{lem:Q-density-moments}
For $m\ge s\ge1,$ $$0\le d_m(W)\le a^{m-s}d_s(W)$$
and for $s\ge1,$ $$d_s(W)\le \frac13 a^{s-1}.$$
\end{lemma}

Among $d_s(W)$ for $s \ge 1$, $d_1(W)$ plays a particularly important role in the proof.
By the normalization \(W=\frac12+U\) and \(C=2U\),
$(C\one)(x)
=
2\int_0^1 W(x,y)\,dy-1.$
Hence
\[
d_1(W)
=
\|C\one\|_2^2
=
4\int_0^1
\left(
\int_0^1 W(x,y)\,dy-\frac12
\right)^2\,dx.
\]
Thus \(d_1(W)/4\) is exactly the \(L^2\)-variance of the outdegree
function $x\longmapsto \int_0^1 W(x,y)\,dy.$ 
We can use \(d_1(W)\) to quantify the deviation of the out-degree function from \(1/2\): \(d_1(W)=0\) exactly when \(W\) is \emph{regular} (i.e. $\int_0^1 W(x,y)\,dy=1/2$ for almost every $x$), while small \(d_1(W)\) means that its out-degree function is close to \(1/2\) in \(L^2\).

Finally, we give the following simple lemma for the oriented tree density in regular tournaments.
\begin{lemma}
\label{lem:regularity}
If \(W\) is regular and \(T\) is an oriented tree, then
$
t(T,W)=2^{-|E(T)|}.
$
\end{lemma}

\begin{proof}
Choose a leaf \(v\) of \(T\), and let \(u\) be its neighbour. If the
edge is directed from \(v\) to \(u\), then for each fixed value of
the variable assigned to \(u\), the integral over the leaf variable
is $\int_0^1 W(y,u)\,dy.$

Since \(W\) is regular and \(W(u,y)+W(y,u)=1\), both the row and
column integrals of \(W\) are \(1/2\). In particular,
\[
\int_0^1 W(y,u)\,dy=\frac12.
\]
The same conclusion holds if the edge is directed from \(u\) to
\(v\), using $\int_0^1 W(u,y)\,dy=\frac12.$

Thus, regardless of the orientation of the leaf edge, deleting the
leaf and integrating its variable contributes exactly a factor
\(1/2\). Iterating this leaf-removal operation until one vertex
remains gives $t(T,W)=2^{-|E(T)|}.$
\end{proof}
\subsection{The expansion over separated even intervals}

We now derive the exact expansion of the normalized path density in
terms of separated even edge intervals. This expansion will be used
both in the small-$d_1$ regime and in the large-$d_1$ regime.

An \emph{edge interval} is a set of consecutive positions in
\(\{1,\dots,k\}\). Two edge intervals are called \emph{compatible}
if they are disjoint and at least one edge position between them
belongs to neither interval. Thus two compatible intervals are
separated by at least one unused edge. Generally, a family of edge intervals is compatible if every two distinct members
are compatible.

\begin{lemma}
\label{lem:interval-expansion}
For every \(\vec{x}\in\{-1,1\}^k\) and every tournamenton \(W\),
\begin{equation}
 2^k t(P_k(\vec{x}),W)
 =
 \sum_{\cF}
 \prod_{I\in\cF}
 (-1)^{|I|/2}
 \left(\prod_{i\in I}x_i\right)
 d_{|I|/2}(W),
 \label{eq:interval-expansion}
\end{equation}
where \(\cF\) ranges over all compatible families of even edge
intervals, including the empty family.
\end{lemma}

\begin{proof}
Write \(W=\frac12+U\) with \(U\) antisymmetric. Expanding every edge
factor in the defining integral of \(t(P_k(\vec{x}),W)\) gives
\[
 t(P_k(\vec{x}),W)
 =
 \sum_{F\subseteq E(P_k)}
 2^{-(k-|F|)}
 t(P_k(\vec{x})\langle F\rangle,U),
\]
where \(P_k(\vec{x})\langle F\rangle\) is the spanning subgraph of the
oriented path with edge set \(F\).

The nontrivial connected components of \(P_k(\vec{x})\langle F\rangle\)
are the maximal consecutive intervals contained in \(F\). Distinct
maximal intervals have disjoint vertex sets and are separated by at
least one unselected edge, so their densities multiply:
\[
 t(P_k(\vec{x})\langle F\rangle,U)
 =
 \prod_I t(P_k(\vec{x})[I],U),
\]
where the product is over the maximal intervals of \(F\) and $P_k(\vec{x})[I]$ denotes the oriented subpath induced on the edge interval \(I\).

We now evaluate the weight of a component occupying an even
interval \(I\) of length \(2s\). By~\cite[Proposition~2.1 and 2.2]{Grzesik+23}, reversing one edge of an oriented
graph changes its density in the antisymmetric kernel \(U\) by a
factor of \(-1\). Every orientation of a path with an odd number of
edges has zero \(U\)-density. Thus we only need to consider even components.

Recall that \(Q_{2s}\) is the orientation of the
\(2s\)-edge path with all edges directed from the central vertex
toward the two endpoints defined in subsection~\ref{subsec:moments}. Its sign vector
\((x_1,\dots,x_{2s})\) satisfies \(\prod_{i=1}^{2s} x_i = (-1)^s.\) For a component $I$ of even length \(2s\), let \(m(I)\) denote the number of edge reversals needed to transform
the orientation \((x_i)_{i\in I}\) into \(Q_{2s}\). Formally, if an
oriented path is obtained from \(Q_{2s}\) by reversing exactly
\(m(I)\) edges, then its density in \(U\) differs from
\(t(Q_{2s},U)\) by the factor \((-1)^{m(I)}\). Since each reversal
changes the sign product of the interval by \(-1\), the parity of 
$m(I)$ satisfies
 $(-1)^{m(I)}
 =(-1)^s \prod_{i\in I} x_i.$
Using \(t(Q_{2s},U)=\|U^s\one\|_2^2\) and
\(d_s(W)=4^s t(Q_{2s},U)\), we obtain
\[
 2^{2s} t(P_k(\vec{x})[I],U)
 =
 (-1)^s
 \left(\prod_{i\in I}x_i\right)
 d_s(W).
\]

Finally, the maximal components of an edge set \(F\) form a
compatible family of intervals. Conversely, every compatible family
of intervals is the family of maximal components of its union.
Removing the odd components proves the expansion
\eqref{eq:interval-expansion}.
\end{proof}

The rest of the paper is organized as follows.

From Section~\ref{sec:small} to Section~\ref{sec:proof-main}, we prove Theorem~\ref{thm:main}.  
We split the proof into two regimes by the value of $d_1(W)$.
In Section~\ref{sec:small} we deal with the case \(d_1(W)\) is small. 
In this case the density of the oriented path can be controlled by $d_1(W)$ when the number of direction flips is small compared with the length of the path.
In Section~\ref{sec:large} we deal with the case \(d_1(W)\) is large. 
We first prove an exponential decay for consistently directed paths using the spectral methods. 
Then we analyze the decomposition at direction flips and apply the decay to each constant-direction piece, while the flips contribute only a bounded factor. 
This yields the desired bound when the path is long enough compared with the number of flips.
Finally, Section~\ref{sec:proof-main}
assembles the two regimes and proves Theorem~\ref{thm:main}.
In Section~\ref{sec:concluding}, we briefly discuss the relationship between direction flips and the tournament Sidorenko property and pose some open problems.

\section{The Small-\(d_1\) Regime}
\label{sec:small}

In this section we treat the regime where \(d_1(W)\) is small.
The following constants and notations
\[
a=\frac4{\pi^2},\qquad
\gamma=1-\frac{a}{1-a^2},\qquad
B_0=\frac{8}{(49/64-a)^2},\qquad
D=\frac4{(1-a)^2},
\]
and
\[
\tau=\frac1{125},\qquad
\vartheta=\gamma-B_0\tau>0, \qquad \delta=d_1(W).
\] are necessary in our proofs.
The main Lemma of this section is listed as follows:  
\begin{lemma}
\label{prop:small-final}
Assume \(\vec{x}\in\{-1,1\}^k\) has \(r\) direction flips with \(k\ge 
K_0(r):=1+\left\lfloor\frac{Dr+1}{\vartheta}\right\rfloor
\). 
If \(0<d_1(W)\le\tau\), then
$
2^k t(P_k(\vec{x}),W)<1.
$
\end{lemma}

The proof uses the interval expansion to write the normalized density as a sum over compatible even intervals, and then builds the path from left to right via a prefix recurrence. We will show that for \(d_1(W)\le \tau\), The single-interval contribution is linear in the spectral moments \(d_s(W)\). The length-two intervals provide the main negative term, while the longer single intervals are controlled by the geometric decay \(d_s(W)\le a^{s-1}d_1(W)\). Direction flips perturb this linear contribution by only \(O(rd_1(W))\). When the path has few direction flips, this linear term dominates, yielding Lemma~\ref{prop:small-final}.

We remark that the idea of studying the dominance of $d_1(W)$ in our approach is similar in spirit to the method of He et al.~\cite[Section 3]{HeManiNieTungWei25+} for local TS/TAS properties. They assume a sufficiently small spectral radius for $U$, which is stronger than our assumption that \(d_1(W)\) is small. 

 We first define the normalized density of prefix of the oriented path $P_k(\vec{x})$.

\begin{definition}
 For each \(1\le j\le k\), the $j$-prefix normalized density is defined as 
\[Z_j(\vec{x})=2^j t(P_j(x_1,\dots,x_j),W),\]
where \(P_j(x_1,\dots,x_j)\) is the oriented path determined by the
first \(j\) coordinates of \(\vec{x}\). 
In particular, we set 
\(Z_0(\vec{x})=Z_{-1}(\vec{x})=1\).   
\end{definition}

By Lemma~\ref{lem:interval-expansion}, 
\(Z_k(\vec{x})\) is a sum over all families of pairwise disjoint, non-adjacent even-length edge intervals. To better estimate the contribution of intervals, for each even interval \(I\subseteq\{1,\dots,k\}\),we define its weight by:
\begin{equation}\label{eq:wI}
w_I=(-1)^{|I|/2}
\left(\prod_{i\in I}x_i\right)
d_{|I|/2}(W).
\end{equation}
Hence, Lemma~\ref{lem:interval-expansion} gives the following expansion for $j\geq 1$:

\begin{equation}
Z_j(\vec{x})
=
\sum_{\cF}
\prod_{I\in\cF}w_I,
\label{eq:Zk-expansion}
\end{equation}
where the sum is over all compatible families of even intervals
contained in \(\{1,\dots,j\}\).

The next lemma gives the basic recurrence for \(Z_j(\vec{x})\), which is the starting point for all later estimates.

\begin{lemma}
\label{lem:prefix-recurrence}
For every \(1\le j\le k\),
\begin{equation}
Z_j(\vec{x})
=
Z_{j-1}(\vec{x})
+
\sum_{s=1}^{\lfloor j/2\rfloor}
w_{[j-2s+1,j]}\,Z_{j-2s-1}(\vec{x}).
\label{eq:prefix-rec}
\end{equation}
\end{lemma}

\begin{proof}
We fix a compatible family \(\cF\). 
If no interval in \(\cF\) has the right endpoint \(j\), then \(\cF\) is entirely contained in
\(\{1,2,\ldots j-1\}\), and its weight is counted by \(Z_{j-1}(\vec{x})\). 
Otherwise, assume 
$I_{s,j}:=[j-2s+1,j]$
is the even interval which contains $j$ as the right endpoint. 
Since \(\cF\) is
compatible, other intervals must be disjoint from and non-adjacent to \(I_{s,j}\). 
Hence, all other intervals lie
in \(\{1,2,\ldots,j-2s-1\}\), and $\cF\setminus \{I_{s,j}\}$ forms a compatible family in $\{1,2,\ldots,j-2s-1\}$. 
Summing over all $\cF$ containing $I_{s,j}:=[j-2s+1,j]$, we will get $\sum_{\cF: I_{s,j}\in \cF} \prod_{I\in \cF\setminus \{I_{s,j}\}}w_I=Z_{j-2s-1}(\vec{x}),$
and 
\[Z_{j}(\vec{x})=Z_{j-1} + \sum_{s=1}^{\lfloor j/2\rfloor}\sum_{\cF: I_{s,j}\in \cF} w_{I_{s,j}}\left(\prod_{I\in \cF\setminus \{I_{s,j}\}}w_I\right) = Z_{j-1}(\vec{x})
+
\sum_{s=1}^{\lfloor j/2\rfloor}
w_{[j-2s+1,j]}\,Z_{j-2s-1}(\vec{x}).\]
\end{proof}

We continue our proof by controlling the growth of density as we recursively delete the last vertex from $P_k(\vec{x})$. 
For $0\le j\le k$, if $Z_{j-1}(\vec x)\ne0$, then we can define the ratio $R_j(\vec x)$ as 
$$
R_j(\vec x)=\frac{Z_j(\vec x)}{Z_{j-1}(\vec x)}.
$$
The next lemma shows that, when \(d_1(W)\) is sufficiently small, the ratio $R_j(\vec x)$ stays close to 1, and further, the ratio is controlled by the sum of $d_1(W)$ and some $w_I$s.

\begin{lemma}
\label{lem:prefix-ratios}
Assume \(0<\delta\le\tau\). 
Then for every
\(0\le j\le k\), \(Z_j(\vec{x})>0\) and $|R_j(\vec{x})-1|\le4\delta$. 
Moreover, 
\begin{equation}\label{xir}
R_j(\vec{x})-1
\le
\sum_{\substack{I:|I| \equiv 0\mod 2\\\max I=j}}w_I
+B_0\delta^2.    
\end{equation}

Here \(\max I\) is the right
endpoint of \(I\).
\end{lemma}

\begin{proof}
We first prove \(Z_j(\vec{x})>0\) and $|R_j(\vec{x})-1|\le 4 \delta $ simultaneously, by induction on \(j\). 
The base cases \(j=0,1\) is clear because
\(Z_{-1}(\vec{x})=Z_0(\vec{x})=Z_1(\vec{x})=1\),  $R_0(\vec{x})=Z_0(\vec{x})/Z_{-1}(\vec{x})=1$ and $R_1(\vec{x})=Z_1(\vec{x})/Z_0(\vec{x})=1$. 
Assume both
assertions hold for all $i<j$. 
Since \(Z_{j-1}(\vec{x})>0\), by dividing  \(Z_{j-1}(\vec{x})\), the equation \eqref{eq:prefix-rec} is turned to 
\begin{equation}\label{eq:R_j-1}
R_j(\vec{x})-1
=
\sum_{s=1}^{\lfloor j/2\rfloor}
w_{[j-2s+1,j]}
\left(\prod_{\ell=j-2s}^{j-1}R_\ell(\vec{x})^{-1}\right).
\end{equation}
By the induction hypothesis, each \(R_\ell(\vec{x})\) lies in
\([1-4\delta,1+4\delta]\) for $\ell\geq 0$.
Hence $
\left|\prod_{\ell=j-2s}^{j-1}R_\ell(\vec{x})^{-1}\right|
\le(1-4\delta)^{-2s}.$
Using \eqref{eq:wI} and Lemma~\ref{lem:Q-density-moments},
$|R_j(\vec{x})-1|
\le
\delta\sum_{s\ge1}a^{s-1}(1-4\delta)^{-2s}.$
By \(a<1/2\) and $1-4\delta\geq 1-4\tau>7/8$, the ratio
\(a/(1-4\delta)^2< 32/49<1\). 
Thus $|R_j(\vec{x})-1|
\le
\frac{\delta}{(1-4\delta)^2-a}
\le
\frac{\delta}{49/64-a}
\le
4\delta.$
Moreover, $R_j(\vec{x})\geq 1-4\delta>7/8$ implies $Z_j(\vec{x})=R_j(\vec{x})Z_{j-1}(\vec{x})>0$, which implies the two statements for $i=j.$

It remains to prove~\eqref{xir}.
By~\eqref{eq:R_j-1}, we have
\[
R_j(\vec{x})-1-\sum_{s=1}^{\lfloor j/2\rfloor}w_{[j-2s+1,j]}
=
\sum_{s=1}^{\lfloor j/2\rfloor}
w_{[j-2s+1,j]}
\left(\prod_{\ell=j-2s}^{j-1}R_\ell(\vec{x})^{-1}-1\right).
\]
Since $R_\ell(\vec{x})$ lies in
\([1-4\delta,1+4\delta]\) for each $\ell$, 
$
\left|\prod_{\ell=j-2s}^{j-1}R_\ell(\vec{x})^{-1}-1\right|\le(1-4\delta)^{-2s}-1\le 8s\delta(1-4\delta)^{-(2s+1)},$
where the last inequality follows from the mean-value theorem applied to $x\mapsto(1-x)^{-2s}$.
By \(|w_{[j-2s+1,j]}|\le\delta a^{s-1}\),
\((1-4\delta)^2-a\ge49/64-a\) and the equality \(\sum_{s\ge1}s a^{s-1}(1-4\delta)^{-(2s+1)}=(1-4\delta)/((1-4\delta)^2-a)^{2}\),
\[
\begin{aligned}
&\sum_{s=1}^{\lfloor j/2\rfloor}
w_{[j-2s+1,j]}
\left(\prod_{\ell=j-2s}^{j-1}R_\ell(\vec{x})^{-1}-1\right)\le
\delta\sum_{s=1}^{\lfloor j/2\rfloor}a^{s-1}
\left|\prod_{\ell=j-2s}^{j-1}R_\ell(\vec{x})^{-1}-1\right|\\
&\le
8\delta^2
\sum_{s=1}^{\lfloor j/2\rfloor}s a^{s-1}(1-4\delta)^{-(2s+1)}\le
8\delta^2
\sum_{s\ge1}s a^{s-1}(1-4\delta)^{-(2s+1)}
\\
&=
\frac{8\delta^2(1-4\delta)}{((1-4\delta)^2-a)^2}\le
\frac{8\delta^2}{(49/64-a)^2}
=
B_0\delta^2.
\end{aligned}
\]
Therefore, 
\begin{equation*}
R_j(\vec{x})-1\le \sum_{s=1}^{\lfloor j/2\rfloor}w_{[j-2s+1,j]}+B_0\delta^2\le \sum_{\substack{I:|I|\equiv0\mod 2\\\max I=j}}w_I+B_0\delta^2.
\end{equation*}
\end{proof}

\begin{corollary}
\label{cor:log-prefix-bound}
Under the assumptions of Lemma~\ref{lem:prefix-ratios},
\[
\log Z_k(\vec{x})
\le
\sum_{\substack{I\subseteq\{1,\ldots,k\}\\
I:|I|\equiv 0\mod 2}}w_I
+B_0k\delta^2.
\]
\end{corollary}

\begin{proof}
Since \(Z_k(\vec{x})=\prod_{j=1}^kR_j(\vec{x})\) and all
\(R_j(\vec{x})>0\), $\log Z_k(\vec{x})
=
\sum_{j=1}^k\log R_j(\vec{x})
\le
\sum_{j=1}^k(R_j(\vec{x})-1).$
This is because \(\log y\le y-1\) for \(y>0\). By
Lemma~\ref{lem:prefix-ratios}, $R_j(\vec{x})-1
\le
\sum_{\substack{I:|I| \equiv 0\mod 2\\\max I=j}}w_I+B_0\delta^2.$
Summing over \(j\), and using that every even interval has a unique
right endpoint proves this corollary.
\end{proof}

Compared with $R_j(\vec{x})\in[1-4\delta, 1+4\delta]$, the equation~\eqref{xir} gives a sharper estimation: the total error is linear in \(k\)
and quadratic in \(\delta\).
It remains to bound \(\sum_{I:|I| \equiv 0\mod 2} w_I\) for paths
with few direction flips.

\begin{lemma}
\label{lem:linear-bound}
If $\vec{x}\in \{-1,1\}^k$ has \(r\) direction flips, then
\[
\sum_{\substack{I:|I| \equiv 0\mod 2}}w_I
\le(-\gamma k+1+Dr)\delta.
\]
\end{lemma}

\begin{proof}
We first estimate $\sum_{\substack{I:|I| \equiv 0\mod 2}}w_I$ for 
\(\vec{1}=(1,\dots,1)\).
Let $w_I^{\rm const}:=(-1)^{|I|/2} d_{|I|/2}(W)$ be the corresponding
weight of $I$. 
The
corresponding sum is
\[
\sum_{\substack{I:|I| \equiv 0 \mod 2}}w_I^{\rm const}
=
\sum_{s=1}^{\lfloor k/2\rfloor}
(-1)^s d_s(W)(k-2s+1).
\]
We discard all negative terms except for $s=1$, and use \(d_{2j}(W)\le a^{2j-1}\delta\). Then
\begin{equation}\label{eq:wconstant}
\sum_{\substack{I:|I| \equiv 0 \mod 2}}w_I^{\rm const}
\le-(k-1)\delta+k\sum_{j\ge1}a^{2j-1}\delta
=
(-\gamma k+1)\delta.
\end{equation}
To continue our proof, we compare the sum of $w_I$ of $\vec{x}$ with that of $\vec{1}.$
For each interval $I$ of length $2s$, $\prod_{i\in I}x_i=-1$ only if $I$ contains a direction flip.
On the other hand, for each direction flip, the number of $2s$-length interval containing it is at most $(2s-1)$.
Therefore, if \(\vec{x}\)
has at most \(r\) direction flips, then for every \(1\le s\le \lfloor k/2\rfloor\),
\begin{equation}\label{eq:compare s}
\left|
\sum_{\substack{|I|=2s}}\prod_{i\in I}x_i
-(k-2s+1)
\right|
\leq 2\cdot (2s-1)\cdot r
\le 4rs.
\end{equation}
Applying \eqref{eq:compare s} and the bound
\(d_s(W)\le a^{s-1}\delta\) gives
\begin{equation*}
\begin{aligned}
&\left|
\sum_{\substack{|I|=2s}}w_I
-\sum_{\substack{|I|=2s}}w_I^{\rm const}
\right|=
d_s(W)
\left|
\sum_{\substack{|I|=2s}}\prod_{i\in I}x_i
-(k-2s+1)
\right|\le
4rs\,d_s(W)
\le
4rs\,a^{s-1}\delta,
\end{aligned}
\end{equation*}
and 
\begin{equation}\label{eq:compare w}
\begin{aligned}
\left|
\sum_{I:|I| \equiv 0\mod 2}w_I
-\sum_{I:|I| \equiv 0\mod 2}w_I^{\rm const}
\right|
&\le
\sum_{s=1}^{\lfloor k/2\rfloor}
\left|
\sum_{|I|=2s}w_I
-\sum_{|I|=2s}w_I^{\rm const}
\right|
\le
\sum_{s\ge1} 4rs\,a^{s-1}\delta\\
&
=
4r\delta\sum_{s\ge1}s\,a^{s-1}
=
4r\delta\cdot \frac{1}{(1-a)^2}
=
Dr\delta,
\end{aligned}
\end{equation}
where we used the equality
\(\sum_{s\ge1}s a^{s-1}=(1-a)^{-2}\) for \(0<a<1\). Combining the equations \eqref{eq:wconstant} and \eqref{eq:compare w} yields
\[
\sum_{\substack{I:|I| \equiv0 \mod 2}}w_I
\le
(-\gamma k+1+Dr)\delta .
\]
\end{proof}

We are now ready to prove the main Lemma of this
section.

\begin{proof}[Proof of Lemma~\ref{prop:small-final}]
Set \(\delta=d_1(W)\). By Corollary~\ref{cor:log-prefix-bound}, Lemma~\ref{lem:linear-bound} and the condition $\delta\leq\tau$,
$$
\begin{aligned}
\log Z_k(\vec x)
&\le
(-\gamma k+1+Dr)\delta+B_0k\delta^2\le
(-\gamma k+1+Dr+B_0k\tau)\delta=
(-\vartheta k+1+Dr)\delta.
\end{aligned}
$$
The condition \(k\ge K_0(r)\) implies
\(k>(Dr+1)/\vartheta\), so
$-\vartheta k+1+Dr<0.$
Hence the right-hand side is strictly negative, which implies
\(Z_k(\vec{x})=2^k t(P_k(\vec{x}),W)\)<1.
\end{proof}

\section{The Large-\(d_1\) Regime}
\label{sec:large}

In this section we deal with the regime where \(d_1(W)>\tau\). 

Define the constants
\begin{equation}
 b=\frac{\tau}{1+a},\qquad
 \omega=1+\frac{9}{10} b,\qquad
 B_1=\frac1{1-\omega(1-b)},\qquad
 H=\frac{2\omega^4}{3(1-a\omega^2)^2}.
 \label{eq:large-constants}
\end{equation}
Since \(\tau=1/125\le1/32\) and $a=4/\pi^2$, we have \(a\omega^2<1\).

We will prove that 
\begin{lemma}
\label{cor:large-bound}
Assume $\vec{x}\in\{-1,1\}^k$ has $r$ direction flips with 
$
 k \ge K_1(r)
 :=
 1+\left\lfloor
 r+\frac{(r+1)\log B_1+r\log(1+H)}{\log \omega}
 \right\rfloor.$
If $d_1(W)>\tau$, then $
 2^k t(P_k(\vec{x}),W)< 1.$
\end{lemma}

The proof has two parts. First, under the condition $d_1(W)>\tau$, we use the spectral method to prove exponential decay for consistently directed paths in Subsection~\ref{subsec:decay}. Second, in~\ref{subsec:decompositon}, we reduce a general path with $r$ flips to at most \(r+1\) consistently directed segments by removing one edge after each flip together with adjacent intervals. The directed decay then applies to each segment, and the removed part contributes only a factor depending on \(r\), which gives the desired bound.

\subsection{Exponential decay for consistently directed paths}\label{subsec:decay}

Let $W$ be a tournamenton. For the consistently directed \(n\)-edge path \(\vec P_n\), by the definition of $Z_k(\vec{x})$, 
 $Z_n(\vec{\bm{1}_n})=2^n t(\vec P_n(\vec{\bm 1_n}),W)
 =\langle\one,K^n\one\rangle,$
where \(K=2W\) and $\vec{\bm{1}_n}:=(1,1,...,1)\in \{-1,1\}^n$. Since \(W\ge0\), all \(Z_n\) are nonnegative. Recall that $C=2U$ and $A=-C^2$ is a compact, self-adjoint positive semidefinite operator. For
real \(z \ge 0\), define
\begin{equation}
 G(z)=\int_0^a\frac{d\mu_A(\lambda)}{1+\lambda z^2},
 \label{eq:G}
\end{equation}
where \(\mu_A\) is the probability measure derived by the spectral of $A$ from
Lemma~\ref{lem:spectral-integration}.

\begin{lemma}
\label{lem:spectral-gap}
If \(d_1(W)>\tau\), then the spectral radius \(\rho=\rho(K)\) of
\(K\) satisfies $\rho<\frac{1}{\omega}$.
\end{lemma}

\begin{proof}
Recall that \(K=2W=J+C\), where \(J\) is the
rank-one operator \(Jg=\langle g,\one\rangle\one\) and \(C\) is
skew-adjoint. The kernel of \(K\) is nonnegative, so \(K\) is a
positive compact operator.

If \(\rho=0\), then \(\rho<1/\omega\) holds trivially.
Assume \(\rho>0\). 
By
Lemma~\ref{thm:positive-compact}, there exists a non-zero function \(f\ge0\) in $L^2$ with \(Kf=\rho f\). 
From \(K=J+C\) we get $(\rho I-C)f=Jf=\langle f,\one \rangle\one.$
By Lemma~\ref{lem:skew-invertibility}, \(\rho I-C\) is invertible and $(\rho I-C)^{-1}=(\rho I+C)(\rho^2I-C^2)^{-1}$.
So $f=\langle f,\one \rangle(\rho I-C)^{-1}\one.$
By $\langle f,\one \rangle>0$, we have
\begin{equation}\label{eq:1=<>}
1=\frac{1}{\langle f,\one \rangle}\langle \one, f \rangle=\langle\one,(\rho I-C)^{-1}\one\rangle.
\end{equation}
Since \(C^*=-C\), the operator \(C^2\) is self-adjoint.
Hence 
\((\rho^2I-C^2)^{-1}\) is self-adjoint and commutes with \(C\).
Therefore, 
$
\bigl(C(\rho^2I-C^2)^{-1}\bigr)^*
=
(\rho^2I-C^2)^{-1}C^*
=
-(\rho^2I-C^2)^{-1}C
=
-C(\rho^2I-C^2)^{-1},$ implying that \(C(\rho^2I-C^2)^{-1}\) is skew-adjoint. Thus, 
$
\langle \mathbf 1, C(\rho^2I-C^2)^{-1}\mathbf 1\rangle
=
-\langle \mathbf 1, C(\rho^2I-C^2)^{-1}\mathbf 1\rangle=0,
$
and further
\begin{equation}\label{eq:phoI-C}
\langle \mathbf 1,(\rho I-C)^{-1}\mathbf 1\rangle
=
\rho\langle \mathbf 1,(\rho^2I-C^2)^{-1}\mathbf 1\rangle .
\end{equation}
Because \(C^2=-A\), we have \(A\) is compact, positive semidefinite and self-adjoint. 
Writing \(\rho^2I-C^2=\rho^2I+A\) and taking $f=\frac{1}{\rho^2+\lambda}$ in Lemma~\ref{lem:spectral-integration},
we obtain
\begin{equation}\label{eq:phoI+A}
\langle
\one,(\rho^2I+A)^{-1}\one\rangle
=
\int_0^a
\frac{d\mu_A(\lambda)}{\rho^2+\lambda}.
\end{equation}
Combining \eqref{eq:1=<>}, \eqref{eq:phoI-C} and \eqref{eq:phoI+A}, $$1=\langle \mathbf 1,(\rho I-C)^{-1}\mathbf 1\rangle
=
\rho\langle \mathbf 1,(\rho^2I-C^2)^{-1}\mathbf 1\rangle=\rho \langle \one,(\rho^2I+A)^{-1}\one\rangle=\rho\int_0^a
\frac{d\mu_A(\lambda)}{\rho^2+\lambda}=\frac{1}{\rho}G(1/\rho).$$

We now analyze the behavior of $G(z)$ and $zG(z)$.
Using the definition of \(G\) and the fact that \(\mu_A\) is a probability measure,
$1-G(1)
=
\int_0^a \frac{\lambda}{1+\lambda}\,d\mu_A(\lambda).$
Since \(\lambda\le a\) on the support of \(\mu_A\), $\frac{\lambda}{1+\lambda}\ge\frac{\lambda}{1+a}.$
Hence
\begin{equation}\label{eq:G(1)vsb}
1-G(1)
\ge
\frac1{1+a}\int_0^a \lambda\,d\mu_A(\lambda)
=
\frac{d_1(W)}{1+a}
>
\frac{\tau}{1+a}
=
b,
\end{equation}
where the last inequality follows from our assumption $d_1(W)>\tau$. 
For \(0\le z\le \omega\), the derivative $
\frac{d}{dz}\left(\frac{z}{1+\lambda z^2}\right)
=
\frac{1-\lambda z^2}{(1+\lambda z^2)^2}$
is positive, since \(\lambda z^2\le a\omega^2<1\). Therefore $zG(z)$ is strictly increasing on \([0,\omega]\). On the other hand, for each \(\lambda\ge0\), the integrand \(1/(1+\lambda z^2)\) is non-increasing in \(z\ge0\), so \(G\) is non-increasing.
Using \eqref{eq:G(1)vsb}, we have
\(G(1)\le1-b\), and

\begin{equation}\label{eq:RG(R)<1}
\omega G(\omega)\le \omega G(1)\le \omega (1-b)=1-\frac{1}{10}b-\frac{9}{10}b^2<1.
\end{equation} 
This shows that for each $z\in [0,\omega],$
$zG(z)\leq \omega G(\omega)<1$.

Now suppose for contradiction that \(\rho\ge1/\omega\). 
Then
\(1/\rho\in(0,\omega]\), which means $(1/\rho)G(1/\rho)<1.$ But note that $1=\frac1\rho G(1/\rho),$ a
contradiction. Therefore \(\rho<1/\omega\).
\end{proof}

\begin{lemma}
\label{lem:directed-decay}
If \(d_1(W)>\tau\), then for every \(n\ge1\),
\[
 Z_n(\vec{\bm 1_n})\le B_1\omega^{-n}.
\]
The same bound holds for the consistently directed path with all
edges reversed.
\end{lemma}

\begin{proof}
By Lemma~\ref{lem:spectral-gap}, \(\rho(\omega K)<1\). 
Hence by Lemma~\ref{lem:Neumann}, the Neumann
series $(I-\omega K)^{-1}=\sum_{i\ge0}(\omega K)^i$
converges in operator norm. We apply $(I-\omega K)^{-1}=\sum_{i\ge0}(\omega K)^i$
to the constant function \(\one\), and then take the inner product
with \(\one\). Since the series converges in operator norm and the
inner product is continuous, we may interchange summation and inner
product:
\[
\begin{aligned}
\langle\one,(I-\omega K)^{-1}\one\rangle
&=
\sum_{i\ge0}\langle\one,(\omega K)^i\one\rangle=\sum_{i\ge0}\omega ^i\langle\one,K^i\one\rangle.
\end{aligned}
\]
Note that 
$Z_n(\vec{\bm 1_n})=\langle\one,K^n\one\rangle.$
Therefore
$
\sum_{i\ge0}Z_i(\vec{\bm 1_i})\omega^i
=
\langle\one,(I-\omega K)^{-1}\one\rangle.
$
To evaluate the right-hand side, let \(B_\omega=I-\omega C\) and
\(h=B_\omega^{-1}\one\). Lemma~\ref{lem:skew-invertibility} ensures that
\(B_\omega\) is invertible.

We first compute \(\langle h,\one\rangle\). Since
\(h=(I-\omega C)^{-1}\one\), we have $
\langle h,\one\rangle
=
\langle \one,(I-\omega C)^{-1}\one\rangle .$
Using $(I-\omega C)^{-1}=(I+\omega C)(I-\omega^2C^2)^{-1},$ we get
\begin{equation}\label{eq:<h,1>}
\begin{aligned}
\langle h,\one\rangle
&=
\langle (I-\omega C)^{-1}\one,\one\rangle
=
\langle (I+\omega C)(I-\omega^2C^2)^{-1}\one,\one\rangle\\
&=
\langle (I-\omega^2C^2)^{-1}\one,\one\rangle
+
\omega\langle C(I-\omega^2C^2)^{-1}\one,\one\rangle.
\end{aligned}
\end{equation}
Since \(C^*=-C\), the
operator \(C^2\) is self-adjoint, so
\((I-\omega^2C^2)^{-1}\) is self-adjoint and commutes with
\(C\). Apply a similar computation to $(I-\omega^2C^2)^{-1}$ as was done for $(\rho^2I-C^2)^{-1}$ in the proof of Lemma~\ref{lem:spectral-gap}, 
$\langle C(I-\omega^2C^2)^{-1}\one,\one\rangle
=
0$ and $
\langle h,\one\rangle
=
\langle (I-\omega^2C^2)^{-1}\one,\one\rangle
=
\langle \one,(I-\omega^2C^2)^{-1}\one\rangle .$
Writing \(C^2=-A\) and applying Lemma~\ref{lem:spectral-integration} for $f=\frac{1}{1+\omega^2\lambda}$,
\[
\langle \one,(I+\omega^2A)^{-1}\one\rangle
=
\int_0^a\frac{d\mu_A(\lambda)}{1+\lambda \omega^2} = G(\omega).
\]
Thus \(\langle h,\one\rangle=G(\omega)\). By~\eqref{eq:RG(R)<1},
\(1-\omega\langle h,\one\rangle=1-\omega G(\omega)>0\).

Next, we estimate $(I-\omega K)^{-1}$.
By \(K=J+C\), 
\(I-\omega K=I-\omega C-\omega J=B_\omega -\omega J\).
Since \(Jg=\langle g,\one\rangle\one\), Lemma~\ref{lem:rank-one} implies
$B_\omega^{-1}Jg=\langle g,\one\rangle h
=(h\otimes\one)g.$
Hence
$
I-\omega K=B_\omega-\omega J
=
B_\omega (I-\omega B_\omega^{-1}J)
=
B_\omega(I-\omega\,h\otimes\one).
$
Since \(1-\omega\langle h,\one\rangle>0\), Lemma~\ref{lem:rank-one} and 
\((h\otimes\one)h=\langle h,\one\rangle h\) implies
\begin{equation}\label{eq:I-RK}
(I-\omega K)^{-1}\one
=
(I-\omega\,h\otimes\one)^{-1}h
=
\frac{h}{1-\omega\langle h,\one\rangle}=\frac{h}{1-\omega G(\omega)}.
\end{equation}
Hence 
\[\langle\one,(I-\omega K)^{-1}\one\rangle
=\frac{\langle \one, h\rangle}{1-\omega G(\omega)}=
\frac{G(\omega)}{1-\omega G(\omega)}\le
\frac1{1-\omega(1-b)}
=
B_1,\]
where the last inequality is because \(G(\omega)\le1\) and \(1-\omega G(\omega)\ge1-\omega(1-b).\) 
This implies that for each integer $n$, \(Z_n(\vec{\bm 1_n})\omega^n\leq \sum_{i\geq 0}Z_i(\vec{\bm 1_i})\omega^i\le B_1\).
Reversing all edges gives an isomorphic consistently directed path,
so the same bound holds.
\end{proof}

\subsection{Decomposition at direction flips}\label{subsec:decompositon}

We now combine the exponential decay for consistently directed paths
with a combinatorial decomposition to handle general oriented paths
with few direction flips when \(d_1(W)>\tau\).

\begin{proposition}
\label{prop:turn-decomposition}
If $\vec{x} \in \{-1,1\}^k$ has  $r$ direction flips and
\(d_1(W)>\tau\), then
\[
 2^k t(P_k(\vec{x}),W)
 \le B_1^{r+1}(1+H)^r \omega^{r-k}.
\]
\end{proposition}

\begin{proof}
For every direction flip between the edges $j$ and $j+1$, we mark
the edge $j+1$.  Let $\mathcal M$ denote the resulting set of $r$
marked edges.

Recall from Lemma~\ref{lem:interval-expansion} that
\[
2^k t(P_k(\vec x),W)
 =
 \sum_{\mathcal F}
 \prod_{I\in\mathcal F}
 (-1)^{|I|/2}
 \left(\prod_{i\in I}x_i\right)
 d_{|I|/2}(W)= \sum_{\mathcal F}
 \prod_{I\in\mathcal F}w_I,
\]
where $\mathcal F$ ranges over all compatible families of even edge
intervals.  For such a family $\mathcal F$, define
\[
\mathcal E(\mathcal F)
 :=
 \{I\in\mathcal F:\ I\cap\mathcal M\neq\varnothing\}.
\]
Thus $\mathcal E(\mathcal F)$ consists precisely of those selected
intervals which contain at least one marked edge and hence may cross
a direction flip.

We group the terms in the above expansion according to
$\mathcal E(\mathcal F)$.  Fix a compatible family $\mathcal E$ of even intervals, every
member of which meets $\mathcal M$. For an interval $I=[a,b]$, let $
\partial I:=\{a-1,b+1\}\cap [k]$
be the set of its immediate neighbouring edge positions.
Consider the set
$
\mathcal D(\mathcal E)
 :=
 \mathcal M
 \cup
 \bigcup_{I\in\mathcal E} I
 \cup
 \bigcup_{I\in\mathcal E}\partial I.
$
We regard the edges in $\mathcal D(\mathcal E)$ as separators.

The complement $[k]\setminus\mathcal D(\mathcal E)$ is a disjoint union of maximal edge intervals $S_1,\ldots,S_p$. Each connected component of $\mathcal D(\mathcal E)$ contains a marked edge: indeed, for every $I\in\mathcal E$, the interval $I\cup\partial I$ is connected and contains a marked edge, and every marked edge itself lies in $\mathcal D(\mathcal E)$. Let $d$ be the number of connected components of $\mathcal D(\mathcal E)$. Since these components are pairwise disjoint and each contains at least one marked edge, we have $d\le r$. Removing $d$ connected components from the discrete interval $[k]$ leaves at most $d+1$ maximal complementary intervals. Therefore $p\le d+1\le r+1$.
Moreover, no $S_j$ contains a marked edge.  Since every direction
flip has a marked edge immediately after it, the signs of $\vec x$
are constant on each $S_j$.  Thus every $S_j$ supports a
consistently directed path.

We now describe the weight of all compatible families
$\mathcal F$ satisfying $\mathcal E(\mathcal F)=\mathcal E$.
For each possible subfamily \(\mathcal E\) of the form
\(\mathcal E(\mathcal F)\), define the grouped weight
\[
A(\mathcal E)
:=
\sum_{\substack{\mathcal F \text{ compatible}\\
\mathcal E(\mathcal F)=\mathcal E}}
\prod_{I\in\mathcal F}w_I.
\]
Since every compatible family \(\mathcal F\) belongs to exactly one
such class, the exact expansion is partitioned as $$
2^k t(P_k(\vec{x}),W)
=
\sum_{\mathcal E} A(\mathcal E).$$

We now compute \(A(\mathcal E)\) for a fixed \(\mathcal E\). For every compatible family \(\mathcal F\) with
\(\mathcal E(\mathcal F)=\mathcal E\), the remaining intervals
\(\mathcal F\setminus\mathcal E\) contain no marked edge. By
compatibility, they cannot meet any interval of \(\mathcal E\) or
any of its immediate neighbours. Hence each such interval lies
entirely inside one of the segments \(S_j\). Conversely, if we
choose, for each \(j\), an arbitrary compatible family
\(\mathcal F_j\subseteq S_j\), then $
\mathcal F
=
\mathcal E\cup\mathcal F_1\cup\cdots\cup\mathcal F_p$
is a compatible family in \(\{1,\dots,k\}\), and
\(\mathcal E(\mathcal F)=\mathcal E\). This gives a bijection
between the compatible families \(\mathcal F\) with
\(\mathcal E(\mathcal F)=\mathcal E\) and the tuples
\((\mathcal F_1,\ldots,\mathcal F_p)\), where each \(\mathcal F_j\)
is a compatible family of even intervals contained in \(S_j\).

Therefore, for such a family \(\mathcal F\),
\[
\prod_{I\in\mathcal F}w_I
=
\left(\prod_{I\in\mathcal E}w_I\right)
\prod_{j=1}^{p}
\left(\prod_{J\in\mathcal F_j}w_J\right).
\]
Summing over all possible \(\mathcal F\) with
\(\mathcal E(\mathcal F)=\mathcal E\) is thus the same as summing
independently over all choices of the \(\mathcal F_j\). Hence
\[
A(\mathcal E)
=
\sum_{\substack{\mathcal F \text{ compatible}\\
\mathcal E(\mathcal F)=\mathcal E}}
\prod_{I\in\mathcal F}w_I
=
\left(\prod_{I\in\mathcal E}w_I\right)
\prod_{j=1}^{p}
\left(
\sum_{\substack{\mathcal F_j \text{ compatible}\\
\mathcal F_j\subseteq S_j}}
\prod_{J\in\mathcal F_j}w_J
\right).
\]

Because the sign of $\vec x$ is constant on each $S_j$, translating
$S_j$ to $\{1,\ldots,|S_j|\}$ and applying
Lemma~\ref{lem:interval-expansion} gives
$
\sum_{\substack{\mathcal F_j\ \mathrm{compatible}\\
\mathcal F_j\subseteq S_j}}
\prod_{I\in\mathcal F_j}w_I
=
2^{|S_j|}
t\!\left(P_{|S_j|}(\vec{\mathbf1}_{|S_j|}),W\right).
$
Lemma~\ref{lem:directed-decay} therefore implies
$$
\left|
\sum_{\substack{\mathcal F_j\ \mathrm{compatible}\\
\mathcal F_j\subseteq S_j}}
\prod_{I\in\mathcal F_j}w_I
\right|
\le B_1\omega^{-|S_j|}.
$$
If $\mathcal E=\{I_1,\ldots,I_m\}$, with $|I_i|=2\ell_i$, then
$
|A(\mathcal E)|
\le B_1^p\omega^{-L_{\mathcal E}}
\prod_{i=1}^m d_{\ell_i}(W),
$
where $L_{\mathcal E}:=\sum_{j=1}^p|S_j|.$
Since $
L_{\mathcal E}\ge
k-r-\sum_{i=1}^m2\ell_i-2m$
and $p\le r+1,
$
it follows that
\begin{equation}\label{eq:A(E)}
|A(\mathcal E)|
\le B_1^{r+1}\omega^{r-k}
\prod_{I\in\mathcal E}
\left(d_{|I|/2}(W)\omega^{|I|+2}\right).
\end{equation}

Let \(\mathcal E\) range over families of pairwise disjoint, compatible even edge intervals, each of which contains at least one marked edge, and set
\begin{equation}\label{eq:Sigma}
\Sigma
:=
\sum_{\mathcal E}
\prod_{I\in\mathcal E}
d_{|I|/2}(W)\omega^{|I|+2}.
\end{equation}
Recall that $r$ is the number of direction flips of $P_k(\vec{x})$, thus \(|\mathcal M|=r\). For each \(I\in\mathcal E\), assign \(I\) to the leftmost marked edge that it contains. This assignment is injective: if two distinct intervals were assigned to the same marked edge, they would share that edge, contradicting the pairwise disjointness of \(\mathcal E\). Consequently, each marked edge \(m\in\mathcal M\) receives either no interval or one interval from
\[
\mathcal I_m:=\{I:\ I\text{ is an even edge interval and }m\in I\}.
\]
The families arising from \(\mathcal E\) satisfy additional global compatibility and disjointness constraints. If these constraints are dropped and each marked edge chooses independently either the empty set or one interval from \(\mathcal I_m\), the resulting collection of choices is a superset of the encodings of the admissible families \(\mathcal E\). Because \(d_s(W)\ge0\) and \(\omega>1\), every weight \(d_{|I|/2}(W)\omega^{|I|+2}\) is nonnegative; hence enlarging the collection of choices can only increase the sum. Therefore
\[
\Sigma
\le
\prod_{m\in\mathcal M}
\left(
1+\sum_{I\in\mathcal I_m}
d_{|I|/2}(W)\omega^{|I|+2}
\right).
\]

For a fixed marked edge \(m\) and an integer \(s\ge1\), an interval of length \(2s\) containing \(m\) is determined by its left endpoint, which can take at most \(2s\) values (fewer near the boundary of the path). Thus
\[
\sum_{I\in\mathcal I_m}
d_{|I|/2}(W)\omega^{|I|+2}
=
\sum_{s\ge1}
\sum_{\substack{I\in\mathcal I_m\\ |I|=2s}}
d_s(W)\omega^{2s+2}
\le
\sum_{s\ge1}2s\,d_s(W)\omega^{2s+2}.
\]
Hence, using $d_s(W)\leq \frac{1}{3}a^{s-1}$ from Lemma~\ref{lem:Q-density-moments} and \(a\omega^2<1\),
\[
\begin{aligned}
\sum_{s\ge1}2s\,d_s(W)\omega^{2s+2}
\le
\frac23\sum_{s\ge1}s\,a^{s-1}\omega^{2s+2}=
\frac{2\omega^4}{3}\sum_{s\ge1}s(a\omega^2)^{s-1}
=
\frac{2\omega^4}{3(1-a\omega^2)^2}
=H,
\end{aligned}
\]
where the sum computing uses
$
\sum_{s\ge1}s (a\omega^2)^{s-1}=1/(1-(a\omega^2))^2.$

This implies that $(
1+\sum_{I\in\mathcal I_m}
d_{|I|/2}(W)\omega^{|I|+2}
) \leq 1+H$. Since \(|\mathcal M|=r\), we have
$
\Sigma\le(1+H)^r.
$
Combining this estimate with \eqref{eq:A(E)} gives
\[
\begin{aligned}
2^k t(P_k(\vec x),W)
=
\sum_{\mathcal E}A(\mathcal E)
\le
\sum_{\mathcal E}|A(\mathcal E)|
\le
B_1^{r+1}\omega^{r-k}
\sum_{\mathcal E}
\prod_{I\in\mathcal E}
\left(d_{|I|/2}(W)\omega^{|I|+2}\right)
\le
B_1^{r+1}(1+H)^r\omega^{r-k},
\end{aligned}
\]
as claimed.
\end{proof}

Now we are ready to prove Lemma~\ref{cor:large-bound}.

\begin{proof}[Proof of Lemma~\ref{cor:large-bound}]
By Proposition~\ref{prop:turn-decomposition}, $2^k t(P_k(\vec{x}),W)
 \le
 B_1^{r+1}(1+H)^r \omega^{r-k}.$
Set $A_r = r+\frac{(r+1)\log B_1+r\log(1+H)}{\log \omega}.$
Since \(K_1(r)=1+\lfloor A_r\rfloor>A_r\), every integer
\(k\ge K_1(r)\) satisfies \(k>A_r\), which is equivalent to
 $B_1^{r+1}(1+H)^r \omega^{r-k}<1.$
Therefore 
 $2^k t(P_k(\vec{x}),W)<1.$
\end{proof}

\section{Proof of Theorem~\ref{thm:main}}
\label{sec:proof-main}

In this section we complete our proof of Theorem~\ref{thm:main}. 
Recall from Lemma~\ref{prop:small-final} and Lemma~\ref{cor:large-bound} we obtain the small-\(d_1\) threshold
$
 K_0(r):=
 1+\left\lfloor\frac{Dr+1}{\vartheta}\right\rfloor,
$
and the large-\(d_1\) threshold
$
 K_1(r):=
 1+\left\lfloor
 r+\frac{(r+1)\log B_1+r\log(1+H)}{\log \omega}
 \right\rfloor.
 \label{eq:K1-def}
$ Let 
\[
 K(r)=\max\{K_0(r),K_1(r)\}.
 \label{eq:K-def}
\]
We now provide an upper bound of \(K(r)\) which is linear in $r$. 
Using \(\lfloor y\rfloor\le y\) and the
choice of \(\vartheta\), we have
 $K_0(r)\le
 \frac{D}{\vartheta}r+1+\frac1{\vartheta}.$
Similarly, $K_1(r)\le
 \left(
 1+\frac{\log B_1+\log(1+H)}{\log \omega}
 \right)r
 +1+\frac{\log B_1}{\log \omega}.$
Thus, if we define
\begin{equation}
 \alpha=
 \max\left\{
 \frac{D}{\vartheta},
 1+\frac{\log B_1+\log(1+H)}{\log \omega}
 \right\},
 \quad
 \beta=
 1+\max\left\{
 \frac1{\vartheta},
 \frac{\log B_1}{\log \omega}
 \right\},
 \label{eq:alpha-beta}
\end{equation}
then $K(r)\le \alpha r+\beta.$
By the approximate numerical bounds $$\frac{D}{\vartheta}<510.209,\ \ 
1+\frac{\log B_1+\log(1+H)}{\log \omega}<1664.917,$$ and$$1+\frac1{\vartheta}<46.114,\ \ 
1+\frac{\log B_1}{\log \omega}<1453.161,$$
we may take $\alpha=1665$ and $\beta=1454.$
Thus
$K(r)\le 1665r+1454.$
\begin{proof}[Proof of Theorem~\ref{thm:main}]
 Let $P_k(\vec{x})$ be an oriented path having $r$ direction flips of length $k \ge 1665r+1454$. Then \(k\ge K(r)\).
If \(W\) is regular, then $d_1(W)=0$ and  $t(P_k(\vec{x}),W)=2^{-k}$ by 
Lemma~\ref{lem:regularity}.
Otherwise, $W$ is not regular, which implies \(d_1(W)>0\). 
If $0<d_1(W)\le\tau$,  Lemma~\ref{prop:small-final} yields $2^k t(P_k(\vec{x}),W)<1.$
Otherwise,  \(d_1(W)>\tau\), and Lemma~\ref{cor:large-bound}
yields $2^k t(P_k(\vec{x}),W)<1.$
As a consequence, $ t(P_k(\vec{x}),W)\leq 2^{-k}$ always holds, which means \(P_k(\vec x)\) is tournament anti-Sidorenko.
\end{proof}

\iffalse
For the lower bound on \(f(k)\), let \(k\ge\lceil\beta\rceil\) and
set $r_k=\left\lfloor\frac{k-\beta}{\alpha}\right\rfloor.$
Then \(r_k\ge0\) and
 $\alpha r_k+\beta\le k.$
By the first part, every orientation of the \(k\)-edge path with at
most \(r_k\) direction flips is tournament anti-Sidorenko.
Since the number of direction flips is an integer, this is
equivalent to saying that every orientation with fewer than
\(r_k+1\) direction flips is tournament anti-Sidorenko. Therefore
\[
 f(k)\ge r_k+1
 =
 1+\left\lfloor\frac{k-\beta}{\alpha}\right\rfloor.
\]
With \(\alpha=1665,\beta=1454\), this is
\[
 f(k)\ge
 1+\left\lfloor\frac{k-1454}{1665}\right\rfloor .
\]
\fi
\section{Concluding Remarks}\label{sec:concluding}

In this section we give a brief remark on the connection between the number of direction flips and the tournament Sidorenko property.
In this paper we provide a linear condition sufficient for tournament anti-Sidorenko. 
One might hope that there exists a linear condition to force an oriented path to be tournament Sidorenko, while this is false. 
The construction of the counterexample is inspired by the one in~\cite[Corollary 3.12]{HeManiNieTungWei25+}.
Fix $0<c<1$. Let $k$ be sufficiently large, and let $
L=\left\lfloor\frac{k}{\log k}\right\rfloor.
$ 
The path $P$ is constructed as follows: orient the first $L$ edges consistently and the remaining $k-L$ edges alternately. 
Thus the path has at least $k-L-1
=
k\left(1-\frac1{\log k}\right)-O(1)>ck
$
direction flips for all sufficiently large $k$.
Let $T_n$ be the transitive tournament on $[n]$, with $i\to j$
if and only if $i<j$. The $L+1$ vertices in the consistently
directed prefix must map to a strictly increasing sequence, giving
$\binom n{L+1}$ choices. The remaining $k-L$ vertices have at most
$n^{k-L}$ images. Therefore $
|\operatorname{Hom}(P,T_n)|
\le \binom n{L+1}n^{k-L},
$
and hence

$$
t(P,T_n)
\le\frac{\binom n{L+1}}{n^{L+1}}
\le\frac1{(L+1)!}.
$$
Stirling's formula gives $
\log((L+1)!)=(1-o(1))k.
$
Since $1>\log2$, $(L+1)!>2^k$ for sufficiently large $k$.
Consequently,
$$
\limsup_{n\to\infty}t(P,T_n)
\le\frac1{(L+1)!}<2^{-k}.
$$
This proves that for each $0<c<1$, there always exists an oriented path $P$ with $r/k>c$ and is not tournament Sidorenko.

Another direction is to characterize the anti-Sidorenko condition in great detail.
Recall that $f(k)$ is the maximum number $m\in\{0,1,...,k\}$ such that every orientation of the $k$-edge path with fewer than $m$ direction flips is tournament anti-Sidorenko. In this paper, we prove that $\liminf_{k \to \infty}\frac{f(k)}{k}\ge\frac{1}{1665}$. A natural question is whether the limit of $\frac{f(k)}{k}$ exists.

\begin{question}
Does the limit of $\frac{f(k)}{k}$ exist as $k \to \infty$?  
\end{question}

If this limit does not exist, it would be natural to determine the
corresponding lower limit.

\begin{problem}
Determine the exact value of $\liminf_{k \to \infty}\frac{f(k)}{k}$.
\end{problem}

\section*{Declaration on the use of AI}
Part of this project began in 2024. The idea of studying the expansion of the sequence $Z_n(\vec{x})$, the dominance of $d_1(W)$ and the way to deal with decomposition at direction flips was proposed by the authors. GPT‑5.6 Sol and the Eureka system mainly assisted the authors in optimizing the choice of parameters in Section 3 and Section 4 and suggested the spectral method based on the total-cone version of the Krein–Rutman theorem and Neumann series to achieve the desired exponential decay bound for consistently directed paths in Lemma~\ref{lem:directed-decay}, Section~\ref{sec:large}. The mathematical statements and proofs were subsequently verified by the authors, who take full responsibility for the manuscript.

Eureka is a multi-agent system developed by JIUCHONG at the University of Science and Technology of China for mathematical research through human--AI interaction.

\bibliographystyle{plain}
\bibliography{OrientedPaths}

\appendix

\section{Proofs of lemmas in Section~\ref{subsec:spectral}}
\label{app:A}

\subsection{Krein--Rutman Theorem}\label{app:krein-rutman}
In this subsection, we introduce the classical Krein--Rutman Theorem and explain why it implies Lemma~\ref{thm:positive-compact}. The background of the Krein--Rutman Theorem stated here is from~\cite[\S 5.2]{LuLu2022}. 

For a real Banach space \(E\), let
$
E_{\mathbb C}=E+iE=\{u+iv:u,v\in E\}
$
be its complexification, equipped with the norm
$
\|u+iv\|_{E_{\mathbb C}}
=
\sup_{\theta\in[0,2\pi]}
\|u\cos\theta+v\sin\theta\|_E.$
For a bounded real-linear operator \(T:E\to E\), its complexification is
$
T_{\mathbb C}:E_{\mathbb C}\to E_{\mathbb C},$
such that
$T_{\mathbb C}(u+iv)=Tu+iTv.
$
The spectral radius \(\rho(T)\) of a real operator \(T\)
is understood as the spectral radius of its complexification:
\[
\rho(T):=\rho(T_{\mathbb C})
=\sup\{|\lambda|:\lambda\in\sigma(T_{\mathbb C})\}.
\]
Equivalently, \(\rho(T)=\lim_{n\to\infty}\|T^n\|^{1/n}\) by Gelfand's formula.

A nonempty closed convex subset
\(P\subseteq E\) is called a \emph{cone} if
$\alpha P\subseteq P\quad(\alpha>0)$ and
$P\cap(-P)=\{0\}.$
A cone \(P\) induces a partial order on \(E\) by declaring
\(x\le y\) if and only if \(y-x\in P\). A cone is \emph{total} if $\overline{P-P}=E.$
A bounded linear operator \(T:E\to E\) is \emph{positive} with
respect to this order if \(T(P)\subseteq P\).

The total-cone version of the Krein--Rutman theorem is the following.

\begin{theorem}[see~\cite{LuLu2022}, Theorem 5.2]
\label{thm:KR}
Let \(E\) be an ordered Banach space whose positive cone \(P\) is
total. Let \(T:E\to E\) be a compact positive operator. If the
spectral radius \(\rho(T)\) is positive, then \(\rho(T)\) is an
eigenvalue of both \(T\) and \(T^*\), and there exist nonzero vectors
\(u\in P\setminus\{0\}\) and $u^* \in P^*\setminus \{0\}$ such that
\[
Tu=\rho(T)u, \qquad T^*u^*=\rho(T^*)u^*.
\]
\end{theorem}

We now specialize this theorem to \(L^2[0,1]\). Let $P=\{f\in L^2[0,1]: f\ge0\ \text{almost everywhere}\}.$
Then \(P\) is a closed convex cone and \(P\cap(-P)=\{0\}\).
Moreover, \(P\) is total. Indeed, every \(f\in L^2[0,1]\) can be
written as
$f=f_+-f_-,$
where
\[
f_+(x)=\max\{f(x),0\},\qquad
f_-(x)=\max\{-f(x),0\}.
\]
Both \(f_+\) and \(f_-\) lie in \(P\), so \(P-P=E\). Thus every positive compact operator \(T\) on \(L^2[0,1]\)
satisfies the hypotheses of Theorem~\ref{thm:KR}. Hence, if
\(\rho(T)>0\), there exists a nonzero function
\(f\in P\), that is, \(f\ge0\) almost everywhere, such that
$Tf=\rho(T)f.$ This yields the following form used in the main text.

\begin{lemma}[Lemma~\ref{thm:positive-compact} in Section~\ref{sec:preliminaries}]
\label{thm:weak--KR}
Let \(T\) be a positive compact operator on \(L^2[0,1]\). If its
spectral radius \(\rho(T)\) is positive, then there exists a nonzero
function \(f\ge0\) almost everywhere such that
\[
Tf=\rho(T)f.
\]
\end{lemma}

\subsection{Proof of Lemma~\ref{lem:skew-invertibility}}\label{app:skew}
We give a proof of Lemma~\ref{lem:skew-invertibility} here.
\iffalse
\begin{lemma}
\label{lem:appskew-invertibility}
Let \(L^2=L^2([0,1];\mathbb R)\). If $C$ is a bounded operator
satisfies \(C^*=-C\) and \(\eta>0\), then \(\eta I-C\) is invertible, and $\|(\eta I-C)g\|_2^2
=\eta^2\|g\|_2^2+\|Cg\|_2^2$
for $g\in L^2.$
Furthermore,
\[
(\eta I-C)^{-1}
=(\eta I+C)(\eta^2I-C^2)^{-1}.
\]
\end{lemma}
\fi

\begin{proof}[Proof of Lemma~\ref{lem:skew-invertibility}]
Since $C^* = -C$, for every $g \in  L^2$, $\langle Cg,g\rangle = \langle g,C^*g\rangle 
= \langle g,-Cg\rangle 
= -\langle Cg,g\rangle,$
so $\langle Cg,g\rangle = 0$. Hence
\[
\begin{aligned}
\|(\eta I-C)g\|_2^2
&= \langle \eta g - Cg, \eta g - Cg \rangle = \eta^2\|g\|_2^2 - \eta\langle g,Cg\rangle - \eta\langle Cg,g\rangle + \|Cg\|_2^2 = \eta^2\|g\|_2^2 + \|Cg\|_2^2.
\end{aligned}
\]
This gives
$\|(\eta I-C)g\| \ge \eta \|g\|.$
The estimate shows that $\eta I-C$ has closed range. The same
estimate for its adjoint $\eta I+C$ shows that
$\ker(\eta I+C)=\{0\}$. Hence
$
\overline{\operatorname{ran}(\eta I-C)}
=
\ker(\eta I+C)^\perp=L^2.
$
The range is both closed and dense, so $\eta I-C$ is surjective, hence invertible.

Next, $C^2$ is self-adjoint and$
\langle C^2g,g\rangle = \langle Cg,C^*g\rangle 
= \langle Cg,-Cg\rangle 
= -\|Cg\|^2.$
Thus$
\langle (\eta^2 I - C^2)g,g\rangle 
= \eta^2\|g\|^2 + \|Cg\|^2 \ge \eta^2\|g\|^2,$
so $\eta^2 I - C^2 \ge\eta^2I$; hence $0\notin\sigma(\eta^2 I - C^2 )$ and $\eta^2 I - C^2 $ is invertible.

Finally, since
$(\eta I-C)(\eta I+C) = \eta^2 I - C^2,$
we have
$(\eta I-C)(\eta I+C)(\eta^2 I - C^2)^{-1} = I.$
As $\eta I-C$ is invertible, it follows that
$(\eta I-C)^{-1} = (\eta I+C)(\eta^2 I - C^2)^{-1}.$
\end{proof}

\subsection{Proof of Lemma~\ref{lem:Neumann}}\label{app:Neumann}
In this subsection, we provide a proof of a more general version of Lemma~\ref{lem:Neumann}.
\begin{lemma}[Norm-convergent Neumann series]
\label{lem:neumann-proof}
Let \(X\) be a real Banach space and \(S\in\mathcal B(X)\).
If \(\rho(S)<1\), then \(I-S\) is invertible in \(\mathcal B(X)\), and
\[
(I-S)^{-1}=\sum_{n=0}^{\infty}S^n,
\]
where the series converges in operator norm. Moreover, if $\|S\|\le q$ for some $q<1$, then $
\|(I-S)^{-1}\|\le(1-q)^{-1}.
$
\end{lemma}

\begin{proof}
Choose \(q\) with \(\rho(S)<q<1\). By the spectral-radius formula,
$\lim_{n\to\infty}\|S^n\|^{1/n}=\rho(S)<q.$
Hence there exists \(N\ge1\) such that
$\|S^n\|^{1/n}\le q,$ that is $\|S^n\|\le q^n$ for $n\ge N.$

Thus
\[
\sum_{n=0}^{\infty}\|S^n\|
\le
\sum_{n=0}^{N-1}\|S^n\|
+
\sum_{n=N}^{\infty}q^n
<\infty.
\]
Since \(\mathcal B(X)\) is a Banach space, absolute convergence in
operator norm implies convergence. Therefore the series
$
B:=\sum_{n=0}^{\infty}S^n
$
converges in operator norm and defines a bounded linear operator.

For every integer \(m\ge0\), $(I-S)\sum_{n=0}^{m}S^n
=
I-S^{m+1}
=
\left(\sum_{n=0}^{m}S^n\right)(I-S).$
For all sufficiently large \(m\), we have
\(\|S^{m+1}\|\le q^{m+1}\to0\). Letting \(m\to\infty\) and using
the continuity of multiplication in \(\mathcal B(X)\), we obtain
$(I-S)B=I$ and
$B(I-S)=I.$
Hence \(I-S\) is invertible with
\[
(I-S)^{-1}=B=\sum_{n=0}^{\infty}S^n.
\]

If moreover \(\|S\|\le q<1\), then \(\|S^n\|\le q^n\) for every
\(n\ge0\), and therefore
\[
\|(I-S)^{-1}\|
=
\left\|\sum_{n=0}^{\infty}S^n\right\|
\le
\sum_{n=0}^{\infty}\|S^n\|
\le
\sum_{n=0}^{\infty}q^n
=
\frac1{1-q}.
\]
\end{proof}

\end{document}